\documentclass[11pt]{article}

\usepackage{amsthm}
\usepackage{amssymb}
\usepackage{amsmath,enumerate}
\usepackage[ruled, lined, linesnumbered]{algorithm2e}
\usepackage{comment}
\usepackage{thm-restate}
\usepackage[T1]{fontenc}
\usepackage{lmodern}
\usepackage{url} 
\usepackage{hyperref}
\usepackage{graphicx}
\usepackage{subfigure}
\usepackage[noabbrev,capitalise]{cleveref}
\usepackage[affil-it]{authblk}
\usepackage{color}
\usepackage{tikz}
\usetikzlibrary{positioning}
\usepackage[margin=1in]{geometry}

\definecolor{normalColor}{rgb}{0.298,0.447,0.690}
 \definecolor{upperColor}{rgb}{0.867,0.518,0.322}
  \definecolor{withinColor}{rgb}{0.333,0.659,0.408} 
  \definecolor{betweenColor}{rgb}{0.506,0.447,0.702}

\newcommand{\eps}{\varepsilon}
\newcommand{\mc}[1]{\mathcal{#1}}

\renewcommand{\v}{\textup{\textsf{v}}}

\theoremstyle{plain}
\newtheorem{thm}{Theorem}[section]
\newtheorem{lem}[thm]{Lemma}
\newtheorem{claim}{Claim}
\newtheorem{prop}[thm]{Proposition}
\newtheorem{cor}[thm]{Corollary}

\newtheorem{prob}[thm]{Problem}

{\noindent \emph{Proof.} {}{#1}{}}{\hfill
	$\Diamond$\vspace{1em}}

\theoremstyle{plain} 
\newcommand{\thistheoremname}{}
\newtheorem{genericthm}[section]{\thistheoremname}

\theoremstyle{definition}

\DeclareTextAccent{\caronaccent}{T1}{7}

\newcommand{\diam}{diam}

\title{On the maximum degree and order of $K_t$-minor-free graphs with positive Lin--Lu--Yau curvature}
\author{Zi-Xia Song\thanks{Department of Mathematics, University of Central Florida, Orlando, FL 32816, USA. Supported by the Simons Foundation through the Travel Support for Mathematicians Program (SFI-MPS-TSM-00024413). E-mail address: Zixia.Song@ucf.edu}}
 \date{}
 
\begin{document}
\maketitle

\baselineskip 15pt
\begin{abstract}
Motivated by recent results on the order of connected graphs with positive Lin--Lu--Yau Ricci curvature under forbidden minor or forbidden subgraph conditions and minimum degree assumptions, we prove that, for every integer
$t\ge 5$, every locally finite graph $G$ with no $K_t$ minor, minimum degree at least $t-1$ and positive Lin--Lu--Yau Ricci curvature on every edge satisfies
\[
\Delta(G)=O(t^5\log^{3/2}t).
\]
The minimum degree condition $t-1$ is best possible.  Moreover, if $G$ is connected, then $G$ is finite and 
\[|V(G)|<2t\Delta(G)^6=O(t^{31}\log^9 t).
\]

\end{abstract}

\section{Introduction}\label{sec:intro}
A classical theme in Riemannian geometry is that lower bounds on Ricci curvature have strong global consequences. A prototypical result is the Bonnet--Myers theorem~\cite{Bonnet,Myers}: a complete $n$-dimensional
Riemannian manifold with Ricci curvature bounded below by $(n-1)k>0$ has diameter at most $\pi/\sqrt{k}$, where $k$ is a constant, and is therefore compact. Over the past several decades, various notions of Ricci curvature have been developed beyond the smooth setting. These include the curvature-dimension condition of Bakry and \'Emery~\cite{BakryEmery}, the optimal-transport approach of von Renesse and Sturm~\cite{vRS}, Sturm~\cite{Sturm06a,Sturm06b}, and Lott and Villani~\cite{LottVillani},
as well as, in the setting of graphs and Markov chains, the notions of
curvature introduced by Forman~\cite{Forman}, Lin and Yau~\cite{LinYau},
Erbar and Maas~\cite{ErbarMaas}, and Ollivier~\cite{Ollivier07,Ollivier};
see~\cite{NajmanRomon,OllivierSurvey} for surveys.

In this paper we study the curvature introduced by Lin, Lu and Yau~\cite{LLY} as a modification of Ollivier's coarse Ricci curvature.  For a vertex $x$ of a locally finite graph $G$ and $\alpha\in[0,1)$, let $m_x^\alpha$ be the probability measure that puts mass $\alpha$ on $x$ and mass $(1-\alpha)/d(x)$ on each neighbor of $x$. For an edge $xy\in E(G)$, let $\kappa_\alpha(x,y)=1-W(m_x^\alpha,m_y^\alpha)$, where $W$ is the $L^1$-Wasserstein distance with respect to the graph metric. The \emph{Lin--Lu--Yau curvature} of $xy$ is
\[
  \kappa(x,y)=\lim_{\alpha\to 1}\frac{\kappa_\alpha(x,y)}{1-\alpha};
\]
this limit exists~\cite{LLY}, and M\"unch and Wojciechowski~\cite{MW} expressed it without a limit in terms of the graph Laplacian (Theorem~\ref{thm:MW}). We say that $G$ is \emph{positively curved} if $\kappa(x,y)>0$ for every edge $xy$. Lin, Lu and Yau~\cite{LLY} proved a discrete Bonnet--Myers theorem: if $\kappa(x,y)\ge\kappa_0>0$ for every edge $xy$ of a connected locally finite graph $G$, then $\diam(G)\le 2/\kappa_0$, so that $G$ is finite (Theorem~\ref{thm:BonnetMyers}). The graphs attaining equality were studied in~\cite{CKKLMP,LiLu}, and estimates under integral curvature conditions were obtained in~\cite{RamosOlive}.

It is worth noting that positive curvature alone does not imply finiteness: Lu and Wang~\cite{LuWang} constructed infinite, locally finite, positively curved planar graphs. On the other hand, they observed that every edge of a positively curved graph with maximum degree at most $\Delta$ has curvature at least $1/(\Delta(\Delta-1))$ (Lemma~\ref{lem:LWlower}). Together with the discrete Bonnet--Myers theorem, this shows that \emph{a connected locally finite positively curved graph is finite if and only if its degrees are bounded}, in which case its order is bounded in terms of its maximum degree (Lemma~\ref{lem:order}). The central question is therefore which structural conditions bound the maximum degree of positively curved graphs.

This question is an analogue of a problem on the combinatorial curvature of graphs embedded in surfaces. Higuchi~\cite{Higuchi} conjectured that every connected plane graph with minimum degree at least $3$ and positive combinatorial curvature at every vertex is finite. DeVos and Mohar~\cite{DeVosMohar} proved this conjecture and showed that such a graph has at most $3444$ vertices unless it is a prism, an antiprism, or a projective-plane analogue of one of these; Ghidelli~\cite{Ghidelli} improved $3444$ to the optimal value $208$ (see also~\cite{NS,Oldridge}). For the Lin--Lu--Yau curvature, Lu and Wang proved the following analogue.

\begin{thm}[Lu and Wang~\cite{LuWang}]\label{thm:LW}
Let $G$ be a  connected   planar graph with $\delta(G)\ge 3$ and $\kappa(x,y)>0$ for every edge $xy\in E(G)$. Then $\Delta(G)\le 17$  and $|V(G)|\le 17^{544}$.
\end{thm}

The condition $\delta(G)\ge 3$ cannot be dropped~\cite[Fig.~2]{LuWang}. Theorem~\ref{thm:LW} has been followed by results for other classes of graphs. Brooks, Osaye, Schenfisch, Wang and Yu~\cite{BOSWY} and Liu, Lu and Wang~\cite{LiuLuWang} showed that every positively curved connected outerplanar graph with minimum degree at least $2$ has maximum degree at most $9$ and at most $10$ vertices, and that both bounds are sharp; recall that a graph is outerplanar if and only if it has neither $K_4$ nor $K_{2,3}$ as a minor~\cite{CH}. Very recently, Bai, Liu, and the present author~\cite{BLS} extended
  the results on outerplanar graphs to graphs with no $K_{2,t}$ minor. They proved that, for every $t\ge 4$, every positively curved connected graph with minimum degree at least two  and no $K_{2,t}$ minor has maximum degree at most $2t+4$ and at most 
$4(t+2)^4$ vertices. Moreover, in the case $t=3$, they showed that every positively curved connected graph with minimum degree at least two  and no $K_{2,3}$ minor has  at most  $10$ vertices. The positively curved graphs were classified among the generalized Halin graphs~\cite{CLLY} and among the $C_4$-free graphs with minimum degree at least $2$~\cite{LinYou}, and a tight bound was obtained for $\{C_3,C_5\}$-free graphs~\cite{GLLY}. Positively curved regular graphs were studied in~\cite{BJLLLV,Hehl}, and curvature estimates for amply regular graphs, with applications to their diameter and finiteness, were obtained in~\cite{CHLZ,HLX}. After our paper was posted on arXiv, Esperet and
Semaan~\cite{EsperetSemaan2026} independently proved that, for any
fixed integers $s,t$ with $2\le s\le t$, every  positively
curved graph with minimum degree at least $s$ and no
$K_{s,t}$ minor has bounded maximum degree and, hence,
bounded order, with both bounds depending only on $s$ and $t$.  Taking $s=t$ in their theorem and using that  graphs with no $K_{t+1}$ minor  have  no  $K_{t,t}$ minor, their result implies that the graphs in Theorem~\ref{thm:main} (with $t$ replaced by $t+1$) have bounded maximum degree and order, while Theorem~\ref{thm:main} gives explicit polynomial bounds.

\subsection*{Main results}

A graph $H$ is a \emph{minor} of a graph $G$ if $H$ can be obtained from a subgraph of $G$ by contracting edges. We say that a locally finite graph $G$ has a $K_t$ minor if and only if some  finite subgraph of $G$ has a $K_t$ minor.  Graph minors play a central role in graph theory. By Wagner's theorem~\cite{Wagner}, a   graph is planar if and only if it has neither $K_5$ nor $K_{3,3}$ as a minor. The celebrated Graph Minor Theorem of  Robertson and Seymour~\cite{RS} asserts that every class of finite graphs that is closed under taking minors can be described by finitely many forbidden minors. Hadwiger's Conjecture~\cite{Hadwiger}, which asserts that every graph with no $K_t$ minor is $(t-1)$-colorable, is one of the most famous open problems in the area.  Since  planar graphs have no $K_5$ minor and outerplanar graphs have no $K_4$ minor,  the graphs with no $K_t$ minor form a natural generalization of both classes. It is therefore natural to ask whether Theorem~\ref{thm:LW} and the results of~\cite{BOSWY,LiuLuWang} extend to graphs with no $K_t$ minor. The proof of Theorem~\ref{thm:LW} uses the cyclic order of the neighbors of a vertex in a plane embedding, which is not available for minor-closed classes, and the minimum degree condition has to be adjusted.  For integers $a\ge 1$ and $b\ge 0$ let $K^*_{a,b}$ denote the graph obtained from the complete bipartite graph $K_{a,b}$ by adding all edges between the vertices of the part of size $a$. The graphs $K^*_{t-2,n-t+2}$ are the classical extremal graphs for $K_t$ minors; they have no $K_t$ minor, and for $t\le 7$ they have the maximum number of edges among all graphs on $n$ vertices with no $K_t$ minor~\cite{Mader68}.

\begin{prop}\label{prop:sharp}
Let $t\ge 3$ and $n\ge t$. The graph $K^*_{t-2,n-t+2}$ has no $K_t$ minor, minimum degree $t-2$, maximum degree $n-1$, and positive Lin--Lu--Yau curvature on every edge. More precisely, if $X$ and $Y$ denote the parts of sizes $t-2$ and $n-t+2$, then for $x, x'\in X$ with $x\ne x'$ and $y\in Y$, 
\[
  \kappa(x,y)=\frac{t-1}{n-1}  \quad\text{and}\quad \kappa(x,x')=\frac{n}{n-1}.
\]
\end{prop}

Proposition~\ref{prop:sharp} shows that, for every $t\ge 3$, there exist
connected graphs with no $K_t$ minor, minimum degree $t-2$, and positive
Lin--Lu--Yau curvature on every edge whose maximum degree is arbitrarily
large. For $t=5$, this shows that Theorem~\ref{thm:LW} does not extend to
graphs with no $K_5$ minor and minimum degree $3$. Indeed, the graphs
$K^*_{3,s}$ have no $K_5$ minor, have minimum degree $3$, and are positively
curved, but are nonplanar for $s\ge 3$ and have unbounded maximum degree.
Our main result below shows that increasing the minimum degree to $t-1$
suffices to bound the maximum degree. To state it, for $t\ge 3$ let
\[
  c(t)=\min\{c\mid  \text{every graph } G  \text{ with }    |E(G)|\ge c|V(G)| \text{ has a }K_t   \text{   minor}\}.
\]
  Mader~\cite{Mader67} proved that the displayed set indeed has a minimum,
and later showed in~\cite{Mader68} that $c(t)\le 8t\log t$.
Kostochka~\cite{Kostochka} and Thomason~\cite{Thomason84} subsequently
proved that
\[
c(t)=\Theta(t\sqrt{\log t}),
\]
and Thomason~\cite{Thomason01} later determined the asymptotics of $c(t)$.
For $t\le 9$, the exact extremal function is known from results of
Mader~\cite{Mader68}, J{\o}rgensen~\cite{Jorgensen}, and the present author
and Thomas~\cite{SongThomas}, yielding
\[
c(t)=t-2 \qquad \text{for } 3\le t\le 9.
\]  

For $t\le 4$, every graph with minimum degree at least $t-1$ contains a
$K_t$ minor. For $t\le 3$, this is elementary, since every graph with
minimum degree at least $2$ contains a cycle. For $t=4$,
Hadwiger~\cite{Hadwiger} and Dirac~\cite{Dirac} proved that every graph
with minimum degree at least $3$ contains a $K_4$ minor. We prove the following main result for $t\ge5$.

\begin{restatable}{thm}{finite}\label{thm:main}
Let $t\ge 5$, and let $G$ be a locally finite  graph with no $K_t$ minor  such that $\delta(G)\ge t-1$ and
$\kappa(x,y)>0$ for every edge $xy\in E(G)$. Then
\[
\Delta(G)<5\,c(t)^3t^2.
\]
Moreover, if $G$ is connected, then $G$ is finite and 
\[|V(G)|<2t\,\Delta(G)^6<31250\,c(t)^{18}t^{13}.
\]
 
\end{restatable}

 Proposition~\ref{prop:sharp} shows that  the condition $\delta(G)\ge t-1$ in Theorem~\ref{thm:main} cannot be replaced by $\delta(G)\ge t-2$. The class of graphs in Theorem~\ref{thm:main} is not empty; for instance, the graph obtained from $K_{t+1}$ by deleting three pairwise disjoint edges has no $K_t$ minor, minimum degree $t-1$ and Lin--Lu--Yau curvature at least $1$ on every edge. 
The upper bound in Theorem~\ref{thm:main} is probably far from optimal.  Combining Theorem~\ref{thm:main} with the bounds on $c(t)$ gives the following  explicit estimates.
\begin{cor}\label{cor:main}
Let $t\ge 5$, and let $G$ be a connected locally finite graph with no $K_t$ minor such that $\delta(G)\ge t-1$ and $\kappa(x,y)>0$ for every edge $xy\in E(G)$. Then
\begin{enumerate}[(i)]
\item $\Delta(G)=O(t^5\log^{3/2} t)$ and $|V(G)|=O(t^{31}\log^9 t)$;
 
\item $\Delta(G)<5\,(t-2)^3t^2$ and  $|V(G)|<31250(t-2)^{18}t^{13}$\, if\, $5\le t\le 9$.
\end{enumerate}
\end{cor}

\subsection*{Outline of the proof}

We proceed by contradiction. Let $x$ be a vertex of degree
$D\ge 5c(t)^3t^2$. Let $U$ be the set of all vertices other than $x$
that have at least $D/(5c(t))$ neighbors in $N(x)$
(see Figure~\ref{fig:sets}). By the linear bound on the number of edges in graphs with no $K_t$ minor, we have $|U|=O(c(t)^2)$.
A neighbor $y$ of $x$ with fewer than $t-2$ neighbors in $U$ has, by
the minimum degree condition, a neighbor other than $x$  outside  $U$. The key lemma of Lu and Wang (Lemma~\ref{lem:LW1}), which is a
consequence of the formula of M\"unch and Wojciechowski, then shows that $y$ reaches at least about four fifths of $N(x)$ by paths of length two avoiding $x$ and $U$. A counting argument based on a new lemma on common neighbors (Lemma~\ref{lem:pairs}) then shows that there are at most about $D/2$ such vertices $y$.
The remaining neighbors of $x$ either belong to $U$ or have at least
$t-2$ neighbors in $U$. Using pairs of vertices in $U$ with many common neighbors in $N(x)\setminus U$, following an idea of Norin, Postle, and Song~\cite{NPS}, we show that every vertex of the latter kind has two neighbors in $U$ with few common neighbors; otherwise, $G$ contains a $K_t$ minor. Contracting each such vertex to one of these two neighbors then shows that there are at most
$c(t)\binom{t-2}{2}|U|$
vertices of the latter kind. Combining these bounds yields
$D<5c(t)^3t^2$, a contradiction.\medskip

The remainder of the paper is organized as follows.
Section~\ref{sec:prelim} collects the necessary definitions and tools. In Section~\ref{sec:proof}, we prove Theorem~\ref{thm:main}.
Finally, Section~\ref{sec:remarks} contains concluding remarks and open problems.

\section{Preliminaries}\label{sec:prelim}

All graphs in this paper are simple, that is, with no loops or parallel edges.  Graphs are finite unless they are explicitly said to be locally finite; the graphs in the definition of $c(t)$ and in Section~\ref{sec:tw} are finite. Given a graph $G$, for a vertex $v$ we write $N(v)$ for its neighborhood and $d(v)=|N(v)|$. We use $\Delta(G)$ and $\delta(G)$ to denote the maximum and minimum degree of $G$, respectively. 
For $S\subseteq V(G)$ we write 
\[N(S):=\{v\in V(G)\setminus S\mid vu\in E(G) \text{ for some } u\in S\}\]
and $N[S]:=N(S)\cup S$. We use $G[S]$ to denote  the subgraph induced by $S$, that is, the graph with vertex set $S$ and edge set $\{xy\in E(G)\mid x,y\in S\}$.  We write $e(G)$ for the number of edges of  $G$,  and $d(u,v)$ for the distance between $u$ and $v$ in $G$ if $G$ is connected. For any positive integer $k$, we write  $[k]$ for the set $\{1,2, \ldots, k\}$. We use the convention   ``$A:=$'' to mean that $A$ is defined to be the right-hand side of the relation.

\subsection{Lin--Lu--Yau curvature}\label{sec:curv}

Let $G$ be a connected  locally finite graph. 
We use the same notation as in~\cite{LLY}.
A probability distribution (over the vertex set $V(G)$) is a mapping $m: V(G)\to [0,1]$ satisfying $\sum_{x\in V(G)} m(x) = 1$.
Suppose two probability distributions $m_1$ and $m_2$ have finite support. A {\em coupling} between $m_1$ and $m_2$ is a mapping $A: V(G)\times V(G) \to [0,1]$ with finite support so that 
$\sum_{y \in V(G)} A(x,y) = m_1(x) \textrm{ for every $x\in V(G)$ and } \sum_{x\in V(G)} A(x,y) = m_2(y)$ for every $y\in V(G)$. The \textit{Wasserstein-1 transportation distance} between two probability distributions $m_1$ and $m_2$ is defined as follows:
$$W(m_1, m_2) = \inf_A \sum_{x,y\in V(G)} A(x,y) d(x,y),$$
where the infimum is taken over all couplings $A$ between $m_1$ and $m_2$. A random walk $m$ on $G=(V,E)$ is defined as a family of probability measures $\{m_v(\cdot)\}_{v\in V(G)}$ such that $m_v(u) = 0$ for all $u \notin N[v]$. It follows that  $m_v(u) \geq 0$ for all $v,u\in V(G)$ and $\sum_{u\in N[v]} m_v(u) = 1$ for all $v\in V(G)$. 

Let $\binom{V(G)}{2}$ denote the set of all unordered pairs of vertices in $G$. In~\cite{Ollivier}, Ollivier defined the \textit{coarse Ricci curvature} $\kappa: \binom{V(G)}{2} \to \mathbb{R}$ of a vertex pair in $G$ equipped with a random walk $m$ as follows: for distinct vertices $x, y\in V(G)$,  
$$\kappa_{\textrm{Ollivier}}(x,y) = 1 - \frac{W(m_x, m_y)}{d(x,y)}.$$
Note that $W(m_x, m_y)=W(m_y, m_x)$ for any distinct $x,y\in V(G)$, and hence $\kappa_{\textrm{Ollivier}}(x,y)=\kappa_{\textrm{Ollivier}}(y,x)$. One natural choice of random walks in graphs is the \textit{$\alpha$-lazy random walk} ($0\leq \alpha < 1$) defined as 
$$
m_x^{\alpha}(v) = \begin{cases} 
                        \alpha & \textrm{ if $v=x$,}\\
                        (1-\alpha)/d(x) &\textrm{ if $v\in N(x)$,}\\
                        0 & \textrm{ otherwise.}
                    \end{cases}
$$
Lin, Lu and Yau~\cite{LLY} defined
\[
  \kappa(x,y)=\lim_{\alpha\to 1}\frac{\kappa_\alpha(x,y)}{1-\alpha},\qquad\text{where}\quad \kappa_\alpha(x,y)=1-\frac{W_1(m_x^\alpha,m_y^\alpha)}{d(x,y)},
\]
and showed that the limit exists. If $x=v_0,v_1,\dots,v_k=y$ is a shortest path, the triangle inequality for $W_1$ gives $\kappa_\alpha(x,y)\ge \frac1k\sum_{i=1}^k\kappa_\alpha(v_{i-1},v_i)$. Hence, if $\kappa(x,y)>0$ on every edge $xy$ of a connected graph, then $\kappa(x,y)>0$ for all  $x,y\in V(G)$ with $x\ne y$. Consequently,  our hypothesis is the same as being \emph{positively curved} in the sense of~\cite{LuWang}.
We say that $G$ has \emph{positive Lin--Lu--Yau curvature} if 
 $\kappa(x,y)>0$ on every edge $xy$ in $G$. 

M\"unch and Wojciechowski~\cite{MW} gave a limit-free formula for $\kappa(x,y)$. For a function $f\colon V(G)\to\mathbb R$ let $\Delta f(v)=\frac{1}{d(v)}\sum_{u\in N(v)}(f(u)-f(v))$, and call $f$ \emph{$1$-Lipschitz} if $|f(u)-f(v)|\le 1$ for every edge $uv$ in $G$.

\begin{thm}[M\"unch and Wojciechowski~\cite{MW}]\label{thm:MW}
For every edge $xy\in E(G)$,
\[
  \kappa(x,y)=\inf\bigl\{\Delta f(x)-\Delta f(y) \mid  f \text{ is $1$-Lipschitz and } f(y)-f(x)=1\bigr\}.
\]
\end{thm}

Adding a constant to $f$ does not change $\Delta f$, so in Theorem~\ref{thm:MW} we may always assume that $f(x)=0$ and $f(y)=1$. The key lemma of Lu and Wang~\cite[Lemma~1]{LuWang} is stated there for edges $xy$ with $d(x)\ge d(y)$, but its proof does not use this assumption. We state it in the form that we use and include the short proof for completeness.

\begin{lem}[Lu and Wang~\cite{LuWang}]\label{lem:LW1}
Let $xy$ be an edge in a locally finite graph $G$ with $\kappa(x,y)>0$, and let $S$ be a nonempty subset of $N(y)\setminus\{x\}$. Then
\[
  \bigl|N(x)\setminus (N[S]\cup N(y))\bigr| < |S\cap N(x)|+1+d(x)\bigl( 1-|S|/d(y)\bigr).
\]
\end{lem}

\begin{proof}
Let $R:=N(x)\setminus(N[S]\cup N(y))$. Then no vertex in $R$ is adjacent to a vertex in $S\cup\{y\}$ in $G$. Since $S\neq\emptyset$ and $S\subseteq N(y)$, we have $y\in N(S)$.  Then  $y\notin R$. Define $f(v)=1$ for $v\in S\cup\{y\}$, $f(v)=-1$ for $v\in R$, and $f(v)=0$ otherwise.  It can be easily checked that  $f$ is $1$-Lipschitz; moreover $f(x)=0$ and  $f(y)-f(x)=1$. Note that every neighbor of $y$ other than $x$ either lies in $S$ or has value $0$. Therefore, 
\[\Delta f(y)=\frac{-(d(y)-|S|)}{d(y)} \quad\text{ and }\quad\Delta f(x)=\frac{|S\cap N(x)|+1-|R|}{d(x)}.\] By Theorem~\ref{thm:MW},
\[
0< \kappa(x,y)\le \Delta f(x)-\Delta f(y)=\frac{|S\cap N(x)|+1-|R|}{d(x)}+\frac{d(y)-|S|}{d(y)},
\]
which implies that $|R| < |S\cap N(x)|+1+d(x)\bigl( 1- |S|/d(y)\bigr)$, as desired.
\end{proof}

The next two results turn a bound on the maximum degree into a bound on the order, as in~\cite{LuWang}. The first is a lemma of Lu and Wang.

\begin{lem}[Lu and Wang~\cite{LuWang}]\label{lem:LWlower}
Let $G$ be a   locally finite graph with  positive Lin--Lu--Yau curvature and maximum degree $\Delta\ge2$.   Then $\kappa(x,y)\ge 1/(\Delta(\Delta-1))$ for every edge $xy\in E(G)$.
\end{lem}

The proof in~\cite{LuWang} uses the fact that the infimum in Theorem~\ref{thm:MW} is attained by an integer-valued function (see~\cite{BM,BCLMP}). Hence $\kappa(x,y)$ is a positive rational number whose denominator divides $\operatorname{lcm}(d(x),d(y))$, and $\operatorname{lcm}(d(x),d(y))$ is at most $\Delta$ if $d(x)=d(y)$ and at most $\Delta(\Delta-1)$ otherwise.

The second is the discrete Bonnet--Myers theorem of Lin, Lu and Yau.

\begin{thm}[Lin, Lu and Yau~\cite{LLY}]\label{thm:BonnetMyers}
Let $G$ be a  connected  locally finite graph and let $\kappa_0>0$. If $\kappa(x,y)\ge\kappa_0$ for every edge $xy\in E(G)$, then $\diam(G)\le 2/\kappa_0$.
\end{thm}

Combining them as in the proof of  \cite[Theorem 5]{LuWang}, where $\Delta=17$, gives the following bound.
\begin{lem}\label{lem:order}
Let $G$ be a connected locally finite graph with positive Lin--Lu--Yau curvature and maximum degree $\Delta\ge3$. Then $\diam(G)\le 2\Delta(\Delta-1)$ and $|V(G)|\le \Delta^{2\Delta(\Delta-1)}$.
\end{lem}

\begin{proof}
By Lemma~\ref{lem:LWlower} and Theorem~\ref{thm:BonnetMyers}, we have  $\kappa(x,y)\ge 1/(\Delta(\Delta-1))$ for every edge $xy\in E(G)$, and  $\diam(G)\le 2\Delta(\Delta-1)$.  Note that  a vertex in $G$ has at most $\Delta(\Delta-1)^{k-1}$ vertices at distance $k$ and $\Delta\ge3$. Therefore,  
\begin{align*}
  |V(G)|&\le 1+\sum_{k=1}^{\diam(G)}\Delta(\Delta-1)^{k-1}\\
  &=1+\frac{\Delta}{\Delta-2}\bigl((\Delta-1)^{\diam(G)}-1\bigr)\\
  &\le 1+3\bigl((\Delta-1)^{\diam(G)}-1\bigr)\\
  &<3(\Delta-1)^{2\Delta(\Delta-1)}\\
  & \le \Delta^{2\Delta(\Delta-1)},
\end{align*}
where the last step uses  $\bigl(\frac{\Delta}{\Delta-1}\bigr)^{2\Delta(\Delta-1)}=\bigl(1+\frac{1}{\Delta-1}\bigr)^{2\Delta(\Delta-1)}\ge 1+\frac{2\Delta(\Delta-1)}{\Delta-1}=1+2\Delta>3$.
\end{proof}

\subsection{Treewidth and the second smallest eigenvalues}\label{sec:tw}
To prove the finiteness of the graphs in Theorem~\ref{thm:main}, Lemma~\ref{lem:order} suffices. Following the ideas of Bai, Liu and the present author~\cite[Section~5]{BLS}, we use a result of Alon, Seymour and Thomas~\cite{AST} on the treewidth of graphs to obtain a better bound on $|V(G)|$ in Theorem~\ref{thm:main}.

Let $G$ be a finite graph with no isolated vertices, let $A$ be its adjacency matrix, and let $D$ be the diagonal matrix whose diagonal entries are the degrees of the vertices of $G$. The \emph{Laplacian} of $G$ is $L=D-A$, and the \emph{normalized Laplacian} of $G$ is $\mathcal L=I-D^{-1/2}AD^{-1/2}$. Both matrices are symmetric and positive semidefinite, and $0$ is an eigenvalue of both. We denote by $\lambda_2(G)$ and $\mu_2(G)$ the second smallest eigenvalues of $L$ and $\mathcal L$, respectively, counted with multiplicity. The matrix $\mathcal L=D^{1/2}(I-D^{-1}A)D^{-1/2}$ is similar to $I-D^{-1}A$, which is the matrix of the operator $f\mapsto-\Delta f$, where $\Delta f$ is as in Theorem~\ref{thm:MW}. Lin, Lu and Yau proved the following analogue of the Lichnerowicz eigenvalue estimate.

\begin{thm}[Lin, Lu and Yau~\cite{LLY}]\label{thm:LLYeig}
Let $G$ be a  connected graph with at least two vertices. Then $\mu_2(G)\ge\min_{xy\in E(G)}\kappa(x,y)$.
\end{thm}

Since $L=D^{1/2}\mathcal LD^{1/2}$, a theorem of Ostrowski~\cite{Ostrowski} on congruent matrices shows that $\lambda_2(G)=\theta\mu_2(G)$ for some $\theta$ between the smallest and the largest degree of $G$. In particular,
\[
  \lambda_2(G)\ge \delta(G)\,\mu_2(G).
\]

A \emph{tree-decomposition} of $G$ is a pair $(T,\mathcal W)$, where $T$ is a tree and $\mathcal W=(W_s: s\in V(T))$ is a family of subsets of $V(G)$, such that $\bigcup_{s\in V(T)}W_s=V(G)$, every edge of $G$ has both ends in some $W_s$, and $W_{s_1}\cap W_{s_3}\subseteq W_{s_2}$ whenever $s_2$ lies on the path of $T$ between $s_1$ and $s_3$. The \emph{treewidth} $\operatorname{tw}(G)$ of $G$ is the minimum $k$ such that $G$ has a tree-decomposition $(T,\mathcal W)$ with $|W_s|\le k+1$ for every $s\in V(T)$. We use the following two results.

\begin{thm}[Gima, Hanaka, Noro, Ono and Otachi~\cite{GHNOO}]\label{thm:GHNOO}
Let $G$ be a   graph with $n$ vertices and $\Delta(G)\ge 1$. Then
\[
  \operatorname{tw}(G)\ge\frac{n\lambda_2(G)}{\Delta(G)+\lambda_2(G)}-1 .
\]
\end{thm}

\begin{thm}[Alon, Seymour and Thomas~{\cite[(1.8)]{AST}}]\label{thm:AST}
Let $h\ge 1$ be an integer, and let $G$ be a finite graph with $n$ vertices and treewidth at least $h^{3/2}n^{1/2}$. Then $G$ has a $K_h$ minor.
\end{thm}

The following proof follows Bai, Liu and the present author~\cite[Section~5]{BLS}, where the bound $\operatorname{tw}(G)\le t$ for graphs with no $K_{2,t}$ minor plays the role of Theorem~\ref{thm:AST}.

\begin{thm}\label{thm:polyorder}
Let $t\ge 3$, and let $G$ be a connected graph with at least three vertices and no $K_t$ minor such that $\kappa(x,y)>0$ for every edge $xy\in E(G)$. Then
\[
  |V(G)|<\Bigl(t^{3/2}\Bigl(1+\frac{\Delta(G)^2(\Delta(G)-1)}{\delta(G)}\Bigr)+1\Bigr)^2 .
\]
\end{thm}

\begin{proof} 
Let $n:=|V(G)|$, $\Delta:=\Delta(G)$ and $\delta:=\delta(G)$. Since $G$ is connected and $n\ge 3$, we have $\Delta\ge 2$ and $\delta\ge 1$. By Lemma~\ref{lem:LWlower}, $\kappa(x,y)\ge 1/(\Delta(\Delta-1))$ for every edge $xy\in E(G)$. By  Theorem~\ref{thm:LLYeig} and the fact that $\lambda_2(G)\ge \delta(G)\,\mu_2(G)$, we have 
\[
  \lambda_2(G)\ge\delta\,\mu_2(G)\ge\frac{\delta}{\Delta(\Delta-1)}>0 .
\]
By Theorem~\ref{thm:GHNOO},
\[
  n\le\bigl(\operatorname{tw}(G)+1\bigr)\Bigl(1+\frac{\Delta}{\lambda_2(G)}\Bigr)\le\bigl(\operatorname{tw}(G)+1\bigr)K,\qquad\text{where}\quad K=1+\frac{\Delta^2(\Delta-1)}{\delta}.
\]
Since $G$ has no $K_t$ minor, we see that $\operatorname{tw}(G)<t^{3/2}n^{1/2}$ by Theorem~\ref{thm:AST}.  Therefore 
\[n<(t^{3/2}n^{1/2}+1)K.\] 
Suppose that $n^{1/2}\ge t^{3/2}K+1$. Then $n^{1/2}>K$, and so
\[
  n\ge n^{1/2}\bigl(t^{3/2}K+1\bigr)=t^{3/2}Kn^{1/2}+n^{1/2}>\bigl(t^{3/2}n^{1/2}+1\bigr)K,
\]
a contradiction. Hence $n^{1/2}<t^{3/2}K+1$, as desired.
\end{proof}

Combining Theorem~\ref{thm:polyorder} and Theorem~\ref{thm:LW}
leads to the following.

\begin{cor}
Let $G$ be a  connected   planar graph with $\delta(G)\ge 3$ and $\kappa(x,y)>0$ for every edge $xy\in E(G)$. Then   $|V(G)|\le 3\cdot 10^8$. 
\end{cor}

\section{Proof of Theorem~\ref{thm:main}}\label{sec:proof}
We first recall an equivalent definition of an $H$ minor in a graph $G$.
Let $H$ be a graph on $h$ vertices. A collection
$\mc{X}=\{X_1,X_2,\ldots,X_h\}$ of pairwise disjoint nonempty subsets of $V(G)$
is a \emph{model of $H$ in $G$} if $G[X_i]$ is connected for every
$i\in[h]$, and there exists a bijection $\phi:V(H)\to[h]$ such that,
for every $uv\in E(H)$, there is an edge in $G$ between
$X_{\phi(u)}$ and $X_{\phi(v)}$.
It is well known and easy to see that $G$ contains $H$ as a minor if and
only if there exists a model of $H$ in $G$.
It also follows immediately that if $G$ has no $K_t$ minor,  then neither
does any subgraph of $G$; moreover, $G[N(x)]$ has no $K_{t-1}$ minor for every vertex  
$x\in V(G)$. 

We begin with a key lemma on graphs with no $K_t$ minor, which   plays
a crucial role in the proof of Theorem~\ref{thm:main}.

\begin{lem}\label{lem:pairs}
 Let $G$ be a locally finite graph with  no $K_t$ minor and let $x\in V(G)$ with $d(x)=D$. Suppose that every graph with no $K_{t-1}$ minor on $D$ vertices has at most $m_D$ edges.  Let $r\ge 2$ be an integer, and let $\mathcal{P}_r$ be the set of unordered pairs $\{y,y'\}$ of distinct vertices of $N(x)$ that have a common neighbor $z\in V(G)\setminus N[x]$ with  $|N(z)\cap N(x)|\le r$. Then
\[
   |\mathcal{P}_r|\le \frac r2\,\bigl(m_D-e(G[N(x)])\bigr)+e(G[N(x)]).
\]
\end{lem}

\begin{proof} Since $G$ is locally finite, we see that $N(x)$ is finite.  Let $Q:=\mathcal{P}_r\setminus E(G[N(x)])$. For each $q\in Q$,  fix a common neighbor $z(q)\in V(G)\setminus N[x]$ of the two vertices of $q$ such that  $|N(z(q))\cap N(x)|\le r$. Let $Z:=\{z(q)\mid q\in Q\}$. Then $Z$ is finite because every $z\in Z$ has a neighbor in $N(x)$ and $N(x)$ is finite. 
For each $z\in Z$,  let $Q_z:=\{q\in Q\mid z(q)=z\}$. Then the sets $\{Q_z\mid   z\in Z\}$ partition $Q$, and each $Q_z$ is a nonempty set of unordered pairs of vertices of $N(z)\cap N(x)$. Thus  $2\le |N(z)\cap N(x)|\le r$.\medskip 

Let $z\in Z$. Counting the pairs $(v,q)$ with $v\in N(z)\cap N(x)$, $q\in Q_z$ and $v\in q$ in two ways yields
\[
  \sum_{v\in N(z)\cap N(x)}\bigl|\{q\in Q_z\mid v\in q\}\bigr|=|\{(v, q)\mid v\in q \text{ and } q\in Q_z\}|=2|Q_z|.
\]
Since $ |N(z)\cap N(x)|\le r$, it follows that  there exists  a vertex $w_z\in N(z)\cap N(x)$ such that $w_z$ lies in at least $2|Q_z|/|N(z)\cap N(x)|\ge 2|Q_z|/r$ pairs of $Q_z$. Let $M=\{zw_z\mid z\in Z\}$.  Note that  $|M|=|Z|$ and $M$ is not necessarily a matching in $G$. Let $H$ be obtained from $G$ by first contracting every edge $zw_z$ in $M$ onto $w_z$ and then deleting all the other vertices in $V(G)\setminus N[x]$ and any resulting loops and parallel edges. Then $H$ is a minor of   the finite subgraph $G[N[x]\cup Z]$ with $V(H)= N[x]$. Moreover,  $x$ is adjacent to all vertices in $N(x)$ in $H$. Therefore, $H[N(x)]$ contains no $K_{t-1}$ minor. By the assumption, we have $e(H[N(x)])\le m_D$.  On the other hand,  every edge of $G[N(x)]$ is an edge of $H[N(x)]$.  For every $z\in Z$, since  $\{Q_z\mid   z\in Z\}$  are pairwise disjoint and  $w_z $   lies in at least $2|Q_z|/r$ pairs of $Q_z$, we see that contracting the edge $zw_z$    adds at least  $2|Q_z|/r$ new edges incident to $w_z$ in  $H[N(x)]\setminus E(G[N(x)])$.  Let $e:=e(G[N(x)])$. Then 
\[
  m_D\ \ge\ e(H[N(x)])\ \ge\ e+\sum_{z\in Z}\frac{2|Q_z|}r\ \ge\ e+\frac{2}{r}\sum_{z\in Z}|Q_z|\ =\ e+\frac{2}{r}\,|Q|.
\]
It follows that   $|Q|\le\frac r2(m_D-e)$, and so $|\mathcal{P}_r|\le |Q|+e\le\frac r2(m_D-e)+e$, as desired.
\end{proof}

We are now ready to prove Theorem~\ref{thm:main}, which we restate below for convenience.\finite*

\begin{proof} 
Let $t,  G, c(t)$ be as in the statement. Let $c:=c(t)$. By Proposition~\ref{prop:sharp}, $c\ge t-2\ge 3$. Fix a vertex $x$ in $G$. Let  $D:=d(x)$.  We next apply Lemma~\ref{lem:pairs} to show that
\begin{equation}\label{eq:goal}
  D<5c^3t^2 .
\end{equation}
Suppose  $D\ge 5c^3t^2$. 
Let $ e:=e(G[N(x)])$. 
For each $y\in N(x)$, by 
Lemma~\ref{lem:LW1} applied to the edge $xy$  and a nonempty set $S\subseteq N(y)\setminus\{x\}$, we have 
\begin{equation}\label{eq:LW}
  \bigl|N(x)\setminus(N[S]\cup N(y))\bigr| < |S\cap N(x)|+1+D\big(1-|S|/d(y)\big).
\end{equation}
Note that $\sum_{y\in N(x)}|N(y)\cap N(x)|=2e$ because each edge of $G[N(x)]$ is counted twice in   the sum $\sum_{y\in N(x)}|N(y)\cap N(x)|$. Since $G$ has no $K_t$ minor, we see that $G[N(x)]$ has no $K_{t-1}$ minor. By the definition of $c=c(t)$, we have $e\le cD$. Thus 
\begin{equation}\label{eq:e(N(x))}
\sum_{y\in N(x)}|N(y)\cap N(x)|=2e\le 2cD.
\end{equation}
Moreover, for every finite set $A\subseteq V(G)\setminus N[x]$, $G[A\cup N[x]]$ has no $K_t$ minor. By the choice of $c$,  $e(G[A\cup N[x]])\le c(|A|+1+D)$. Note that $e(G[N[x]])= D+e$, and so $e(G[A\cup N[x]])\ge D+e+ \sum_{a\in A}|N(a)\cap N(x)|$. It follows from the two inequalities on $e(G[A\cup N[x]])$ that  
\[\sum_{a\in A}|N(a)\cap N(x)|-c|A|\le c(D+1)-D-e \] and so  

\begin{equation}\label{eq:e(A)}
 \sum_{a\in A}(|N(a)\cap N(x)|-c)\le c(D+1)-e-D \text{ for every  finite  set } A\subseteq V(G)\setminus N[x].
\end{equation}
Let $\eps=1/(5c)$, and let
\[
  U:=\{v\in V(G)\mid  v\ne x \text{ and } |N(v)\cap N(x)|\ge\eps D\} \text{ and }\]
 \[W:=\{y\in N(x)\setminus U\mid |N(y)\cap U|\ge t-2\}.
\]
Since $G$ is locally finite, we see that $N(x)$ is finite and so  $U$ is finite because every vertex of $U$ has a neighbor in $N(x)$. Note  that $\eps D=D/(5c)\ge c^2t^2$. The sets $U$ and $W$, together with the set $Z$  defined below, are illustrated in Figure~\ref{fig:sets}. We next prove several claims.

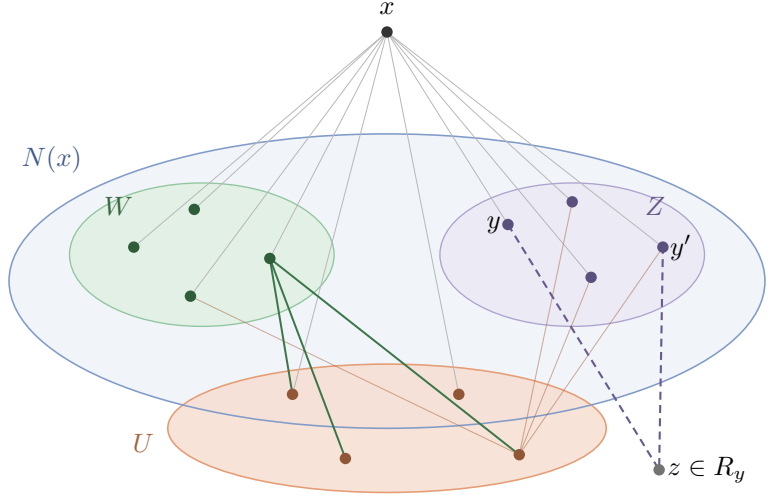
\begin{figure}[thb]
\centering
\begin{tikzpicture}[
    vtx/.style={circle, fill=#1, inner sep=0pt, minimum size=4.3pt},
    vtx/.default=black!80,
    lab/.style={font=\small},
  ]
  \fill[normalColor!8] (0,0) ellipse [x radius=5, y radius=1.95];
  \fill[withinColor!17, draw=withinColor!65, line width=0.5pt] (-2.45,0.35) ellipse [x radius=1.75, y radius=0.95];
  \fill[betweenColor!15, draw=betweenColor!65, line width=0.5pt] (2.45,0.35) ellipse [x radius=1.75, y radius=0.95];
  \fill[upperColor!28, fill opacity=0.8, draw=upperColor!80, line width=0.6pt] (0,-1.95) ellipse [x radius=2.9, y radius=0.85];
  \draw[normalColor!70, line width=0.6pt] (0,0) ellipse [x radius=5, y radius=1.95];
  \coordinate (x)  at (0,3.3);
  \coordinate (w1) at (-3.35,0.45); \coordinate (w2) at (-2.55,0.95);
  \coordinate (w3) at (-1.55,0.3);  \coordinate (w4) at (-2.6,-0.2);
  \coordinate (y)  at (1.6,0.75);   \coordinate (b2) at (2.45,1.05);
  \coordinate (b3) at (2.7,0.05);   \coordinate (yp) at (3.65,0.45);
  \coordinate (u1) at (-1.25,-1.5); \coordinate (u2) at (0.95,-1.5);
  \coordinate (u3) at (-0.55,-2.35);\coordinate (u4) at (1.75,-2.3);
  \coordinate (z)  at (3.6,-2.5);
  \foreach \v in {w1,w2,w3,w4,y,b2,b3,yp,u1,u2}
     \draw[black!28, line width=0.3pt] (x) -- (\v);
  \foreach \v in {w4,b2,b3,yp}
     \draw[upperColor!75!black, line width=0.4pt, opacity=0.5] (u4) -- (\v);
  \foreach \v in {u1,u3,u4}
     \draw[withinColor!70!black, line width=0.75pt] (w3) -- (\v);
  \draw[betweenColor!80!black, line width=0.8pt, dash pattern=on 3pt off 2pt] (y) -- (z) -- (yp);
  \node[vtx] at (x) {};
  \foreach \v in {u1,u2,u3,u4} \node[vtx=upperColor!65!black] at (\v) {};
  \foreach \v in {w1,w2,w3,w4} \node[vtx=withinColor!55!black] at (\v) {};
  \foreach \v in {y,b2,b3,yp} \node[vtx=betweenColor!70!black] at (\v) {};
  \node[vtx=black!55] at (z) {};
  \node[lab, above=3pt] at (x) {$x$};
  \node[lab, text=normalColor!80!black, anchor=south east] at (-3.9,1.3) {$N(x)$};
  \node[lab, text=withinColor!60!black] at (-3.55,1.0) {$W$};
  \node[lab, text=betweenColor!75!black] at (3.55,1.0) {$Z$};
  \node[lab, text=upperColor!75!black, anchor=east] at (-2.95,-2.15) {$U$};
  \node[lab, anchor=east, inner sep=2.5pt] at (y) {$y$};
  \node[lab, anchor=west, inner sep=2.5pt] at (yp) {$y'$};
  \node[lab, anchor=west, inner sep=3pt] at (z) {$z\in R_y$};
\end{tikzpicture}
\caption{$N(x)$ is the disjoint union of $U\cap N(x)$, $W$ and $Z$.}
\label{fig:sets}
\end{figure} 

\begin{claim}\label{cl:U}
$|U|\le 5c^2-4c$.
\end{claim}

\begin{proof}  Let $U_1:=U\cap N(x)$ and $U_2:=U\setminus N(x)$. By the choice of $U$, 
every $v\in U$ satisfies $|N(v)\cap N(x)|\ge\eps D$. In the sum $\sum_{v\in U_1}|N(v)\cap N(x)|$, every edge of $G[N(x)]$ with both ends in $U_1$ is counted twice and every other edge of $G[N(x)]$ is counted at most once. Since $G[U_1]$ has no $K_t$ minor, we have $e(G[U_1])\le c|U_1|$, and so
\[
  e\ \ge\ \sum_{v\in U_1}|N(v)\cap N(x)|-e(G[U_1])\ \ge\ \eps D|U_1|-c|U_1|=(\eps D-c)|U_1|.
\]
Moreover, every $z\in U_2$ satisfies $|N(z)\cap N(x)|-c\ge\eps D-c$. By inequality~\eqref{eq:e(A)} applied to $A:=U_2$, we have
\[
  (\eps D-c)|U_2|\le\sum_{z\in U_2}\bigl(|N(z)\cap N(x)|-c\bigr)\le c(D+1)-e-D.
\]
Combining  the two inequalities yields $(\eps D-c)|U|\le (c-1)D+c$. Since $\eps D-c=(D-5c^2)/(5c)>0$, we have
\[
  |U|\le\frac{5c\bigl((c-1)D+c\bigr)}{D-5c^2}\le 5c^2-4c,
\]
where the last inequality holds because $D\ge 5c^3t^2> 25c^3$, and so $(5c^2-4c)(D-5c^2)-5c\bigl((c-1)D+c\bigr)=c\,(D-25c^3+20c^2-5c)\ge 0$.
\end{proof}

\begin{claim}\label{cl:W}
 $|W|\le c\binom{t-2}{2}|U|$.
\end{claim}

\begin{proof} The proof of this claim follows an idea of Norin, Postle and Song~\cite{NPS}. 
We call  a pair $\{u,u'\}$ of distinct vertices of $U$ \emph{mates} if $u$ and $u'$ have more than $\binom{t-2}{2}$ common neighbors in $N(x)\setminus U$,  otherwise, we call them \emph{unmated}. We first show that no vertex $z\in W$ has $t-2$ neighbors in $U$ that are pairwise mates. Suppose to the contrary that there exists $z\in W$   adjacent to $u_3,\dots,u_t\in U$ such that $u_3,\dots,u_t$ are pairwise mates. 
For $3\le i<j\le t$ we choose, one after another, distinct vertices $z_{ij}\in N(x)\setminus U $ with $z_{ij}\ne z$ such that $z_{ij}$ is adjacent to both $u_i$ and $u_j$. This is possible, since $u_i$ and $u_j$ have at least $\binom {t-2}{2}+1$ common neighbors in $N(x)\setminus U$, of which at most $\binom{t-2}{2}$ are excluded (namely $z$ and  at most $\binom {t-2}{2}-1$ vertices chosen before). Let $Z_1:=\{x\}$, $Z_2:=\{z\}$, $Z_t:=\{u_t, z_{3t}\}$ and $Z_3:=\{u_3\}\cup\{z_{3j}\mid 3<j\le t-1\}$. For $4\le i\le t-1$, let $Z_i:=\{u_i\}\cup\{z_{ij}\mid i<j\le t\}$.  It can be easily checked that each $G[Z_i]$ is connected,  and for $1\le i<j\le t$, there is an edge in $G$ between $Z_i$ and $Z_j$.  Thus $\{Z_1,\ \dots,\ Z_t\}$
 is a model of $ K_t$ in $G$, contrary to  the assumption that $G$ has no $K_t$ minor. This proves that no vertex $z\in W$ has $t-2$ neighbors in $U$ that are pairwise mates. 

Let $z\in W$. By the choice of $W$, $z$ is adjacent to at least $t-2$ vertices in $U$. Choosing any $t-2$ neighbors of $z$ in $U$, we see from the above that $z$ has two distinct neighbors in $U$ that are unmated. We fix such an unordered pair $P_z\subseteq N(z)\cap U$. Let $\mathcal{P}:=\{P_z\mid z\in W\}$. Note that $|\mathcal{P}|\le |W|$ because  distinct vertices $z, z'\in W$ may satisfy $P_z=P_{z'}$.  Let $U:=\{x_1, \ldots, x_{|U|}\}$. For each $1\le i\le |U|-1$, let $X_i:=\{x_i\}\cup\{z\in W\mid P_z=\{x_i, x_j \} \text{ for some } i<j\le |U|\}$. Finally, let $X_{|U|}:=\{x_{|U|}\}$. It can be easily checked that $X_1, \ldots, X_{|U|}$ are pairwise disjoint, and each $G[X_i]$ is connected. Moreover, $G$ has an edge between $X_i$ and $X_j$   if $\{x_i, x_j \}\in \mathcal{P}$. Let $H$ be the graph with vertex set $U$ such that two distinct vertices   $x_i, x_j$ are adjacent in $H$ if and only if $G$ has an edge between $X_i$ and $X_j$. Then $|\mathcal{P}|\le e(H)$ and $X_1, \ldots, X_{|U|}$ is a model of $H$ in $G$. Thus $H$ has no $K_t$ minor and so $|\mathcal{P}|\le e(H)\le c|U|$.  On the other hand, each $\{x_i, x_j \}\in \mathcal{P}$ are unmated and thus  have at most $\binom{t-2}{2}$ common neighbors in $N(x)\setminus U$.  Therefore $|W|\le \binom{t-2}{2}|\mathcal{P}|\le c\binom{t-2}{2}|U|$.
\end{proof}

Let $Z:=N(x)\setminus(U\cup W)$. Then  $N(x)$ is the disjoint union of $U\cap N(x)$, $W$ and $Z$. Thus  $D=|N(x)|\le |U |+|W|+|Z|$. It remains to bound $|Z|$.  For each $y\in Z$,  let 
\[
  R_y:=\{v\in N(y)\mid  v\ne x \text{ and } v\notin U\} \quad\text{and}\quad C_y:=\{v\in N(x)\mid v\ne y \text{ and } N(v)\cap R_y\ne\emptyset\}.
\]
Then  every vertex in $C_y$ has a neighbor in $R_y$ and $N(R_y)\cap N(x)\subseteq C_y\cup\{y\}$.

\begin{claim}\label{cl:Z}
For every $y\in Z$,    $R_y \neq\emptyset$ and $|C_y|>4D/5-t(|N(x)\cap N(y)|+3)$.
\end{claim}

\begin{proof}
Since $y\notin W$ and $\delta(G)\ge t-1$, we have $ |N(y)\cap U|\le t-3$ and $|N(y)|\ge t-1$. Then $|R_y|=|N(y)\setminus\{x\}|-|N(y)\cap U|\ge (t-1-1)-(t-3)=1$, and so $R_y\neq\emptyset$. Moreover, 
\[d(y)-|R_y|=1+|N(y)\cap U|\le t-2.\]
It is worth mentioning that this is the only place where the condition $\delta(G)\ge t-1$ is used to bound $D$ in~\eqref{eq:goal}.  Therefore, the bound $\Delta(G)<5c(t)^3t^2$   holds for the  class of graphs $G$ with $\delta(G)\ge t-1$ and no $K_t$ minor  such that for every  edge  $xy\in E(G)$ and every nonempty $S\subseteq N(y)\setminus\{x\}$, the inequality of Lemma~\ref{lem:LW1} holds.

To bound $|C_y|$, let $z\in R_y$. Then $z\ne x$ and $z\notin U$. By the choice of $U$,   $|N(z)\cap N(x)|<\eps D$. Now applying the inequality~\eqref{eq:LW} to the edge  $xy$ with $S=\{z\}$, together with 
\[|N(x)\cap(N[z]\cup N(y))|\le |N[z]\cap N(x)|+|N(y)\cap N(x)|\le |N(z)\cap N(x)|+1+|N(y)\cap N(x)|\]
and $|\{z\}\cap N(x)|\le 1$, yields
\[
  D-\bigl(|N(z)\cap N(x)|+|N(y)\cap N(x)|+1\bigr)\le |N(x)\setminus(N[z]\cup N(y))|<2+D(1-1/d(y)).\]
  It follows that 
  \begin{equation}\label{eq:Dy}
      \frac{D}{d(y)}<|N(z)\cap N(x)|+|N(y)\cap N(x)|+3<\eps D+|N(y)\cap N(x)|+3.
\end{equation}
Next, since $|R_y\cap N(x)|\le|N(y)\cap N(x)|$ and $d(y)-|R_y|\le t-2$, applying the inequality~\eqref{eq:LW} again to the edge  $xy$   with   $S=R_y$ yields
\[
  \bigl|N(x)\setminus(N[R_y]\cup N(y))\bigr|<|N(y)\cap N(x)|+1+\frac{(d(y)-|R_y|)D}{d(y)}\le|N(y)\cap N(x)|+1+\frac{(t-2)D}{d(y)},
\]
and so  by inequality~\eqref{eq:Dy}
\[
  \bigl|N(x)\setminus(N[R_y]\cup N(y))\bigr|< |N(y)\cap N(x)|+1+(t-2)(\eps D+|N(y)\cap N(x)|+3).
\]
Since $R_y\subseteq N(y)$, every vertex of $N(x)\cap N[R_y]$ lies in $N(y)$ or in $N(R_y)\cap N(x)\subseteq C_y\cup\{y\}$.  Hence 
\[N(x)\subseteq \bigl(N(x)\setminus(N[R_y]\cup N(y))\bigr)\cup(N(x)\cap N(y))\cup C_y\cup\{y\},\]
and so
\[
  D<|N(y)\cap N(x)|+1+(t-2)(\eps D+|N(y)\cap N(x)|+3)+|N(y)\cap N(x)|+|C_y|+1.
\]
It follows that  $|C_y|>D-(t-2)\eps D-t|N(y)\cap N(x)|-3t+4$. As $c\ge t-2$, we have $(t-2)\eps=(t-2)/(5c)\le 1/5$. Therefore, $|C_y|>  D-D/5-t|N(y)\cap N(x)|-3t+4$. It follows that  $|C_y|>4D/5-t(|N(x)\cap N(y)|+3)$.
\end{proof}

\begin{claim}\label{cl:count}
$\displaystyle\sum_{y\in Z}|C_y|\le 2D^2/5$.
\end{claim}

\begin{proof} Let $\mathcal{P}:=\{(y, y')\mid y\in Z \text{ and } y'\in C_y\}$. Then $|\mathcal{P}|=\sum_{y\in Z}|C_y|$. Recall that every vertex in $C_y$ has a neighbor in $R_y$.   For each   pair $(y, y')\in\mathcal{P}$,  choose $z\in R_y$ such that $z$ is adjacent to $y'$. By the choice of $R_y$,  $z\notin U$ and $z\ne x$. By the choice of $C_y$,    $y\ne y'$  and $y, y'\in N(z)\cap N(x)$.  We count the pairs $(y,y')$ by considering separately the two cases $z\in V(G)\setminus N[x]$ and $z\in N(x)$. Let  $\mathcal{P}_1:=\{(y, y')\in\mathcal{P}\mid  z\in V(G)\setminus N[x]\}$ and $\mathcal{P}_2:=\{(y, y')\in\mathcal{P}\mid  z\in  N(x)\}$. Then $|\mathcal{P}|=|\mathcal{P}_1|+|\mathcal{P}_2|$.

Suppose $z\in V(G)\setminus N[x]$. Let $r:=\lfloor\eps D\rfloor$. Then $r\ge 2$ because  $\eps D\ge c^2t^2$. Since $|N(z)\cap N(x)|$ is an integer and $|N(z)\cap N(x)|<\eps D$, we have $|N(z)\cap N(x)|\le r$. Thus  $\{y,y'\}\in \mathcal{P}_r$, where $\mathcal{P}_r$ is defined in Lemma~\ref{lem:pairs}. Note that every graph with no $K_{t-1}$ minor on $D$ vertices has no $K_t$ minor and hence has at most $cD$ edges. By Lemma~\ref{lem:pairs} with $m_D= c(t)D$, we have $|\mathcal{P}_r|\le\frac{r}{2}\,(cD-e)+e$.  Thus $|\mathcal{P}_1|\le2|\mathcal{P}_r|\le r(cD-e)+2e$.

 Suppose next  $z\in N(x)$. Then $|N(z)\cap N(x)|<\eps D$ because $z\notin U$.  Recall that $y, y'\in N(z)\cap N(x)$. For a fixed $z\in N(x)\setminus U$,  there are at most $|N(z)\cap N(x)|(|N(z)\cap N(x)|-1)\le(\eps D-1)|N(z)\cap N(x)|$ such pairs. Summing over all $z\in N(x)\setminus U$ and using $\eps D>1$ and inequality~\eqref{eq:e(N(x))},   we see that  
\[
 |\mathcal{P}_2|\le  (\eps D-1)\sum_{z\in N(x)}|N(z)\cap N(x)|=2(\eps D-1)e.
\]
It follows that 
\[ \begin{aligned} \sum_{y\in Z}|C_y| =|\mathcal{P}|&=|\mathcal{P}_1|+|\mathcal{P}_2|\\ &\le r(cD-e)+2e+2(\eps D-1)e\\ &\le \eps D(cD-e)+2\eps De\\ &=\eps D(cD+e)\\ &\le 2\eps cD^2\\ &=\frac{2D^2}{5},\end{aligned} \]
because   $r\le\eps D$, $0\le e\le cD$ and $\eps c=1/5$. 
\end{proof}

  By inequality~\eqref{eq:e(N(x))}, $\sum_{y\in Z}|N(x)\cap N(y)|\le \sum_{y\in N(x)}|N(x)\cap N(y)|\le 2cD$. This, together with Claim~\ref{cl:Z}, implies that 
\[
  \sum_{y\in Z}|C_y|\ \ge\ |Z|\Bigl(\frac{4D}{5}-3t\Bigr)-t\sum_{y\in Z}|N(x)\cap N(y)|\ \ge\ |Z|\Bigl(\frac{4D}{5}-3t\Bigr)-2tcD.
\]
By Claim~\ref{cl:count}, we have 
$|Z|\bigl(\frac{4D}{5}-3t\bigr)\le\frac{2D^2}{5}+2tcD$. Since $D\ge 5c^3t^2$, we have $3t\le 3D/(5c^3t)$ and $2tc\le 2D/(5c^2t)$. Thus 
\[
  |Z|\ \le\ D\cdot\frac{\frac25+\frac{2}{5c^2t}}{\frac45-\frac{3}{5c^3t}}\ =\ D\cdot\frac{1+\frac{1}{c^2t}}{2-\frac{3}{2c^3t}}\ \le\ D\Bigl(\frac12+\frac{1}{c^2t}\Bigr),
\]
where the last inequality is equivalent to $4c^3t-3c^2t-6\ge 0$, which holds as $c\ge 2$. Recall that  $N(x)$ is the disjoint union of $U\cap N(x)$, $W$ and $Z$. By Claims~\ref{cl:U} and~\ref{cl:W},
\[
  D\ \le\ |U|\Bigl(1+c\binom{t-2}{2}\Bigr)+|Z|\ \le\ (5c^2-4c)\Bigl(1+\frac{c(t-2)(t-3)}{2}\Bigr)+D\Bigl(\frac12+\frac{1}{c^2t}\Bigr).
\]
Thus 
 \[D\Bigl(1-\frac{2}{c^2t}\Bigr)\ \le\ c(5c-4)\bigl(2+c(t-2)(t-3)\bigr).
\]
 However, 
 \[ \begin{aligned} D\Bigl(1-\frac{2}{c^2t}\Bigr)- c(5c-4)\bigl(2+c(t-2)(t-3)\bigr)&\ge 5c^3t^2\Bigl(1-\frac{2}{c^2t}\Bigr)-c(5c-4)\bigl(2+c(t-2)(t-3)\bigr)\\ 
 &=(25t-30)c^3+(4t^2-20t+14)c^2-(10t-8)c\\  
 &>0,\end{aligned} \]
 because  $t\ge 5$ and $c\ge 3$, a contradiction.  This proves that $D<5c^3t^2$ and so $\Delta(G)<5c^3t^2$. \medskip

Assume now that $G$ is connected.  Let $\Delta:=\Delta(G)$ and $\delta:=\delta(G)$. Since $G$ is locally finite with positive curvature on every edge and $\Delta\ge\delta\ge t-1\ge 4$, by Lemma~\ref{lem:order},  $G$ is finite. It remains to bound $|V(G)|$.
    Since $t-1\le\delta\le\Delta\le\Delta^2$, we have
\[
  1+\frac{\Delta^2(\Delta-1)}{\delta}\le 1+\frac{\Delta^2(\Delta-1)}{t-1}\le\frac{\Delta^3}{t-1}.
\]
Moreover, $t^{3/2}/(t-1)\le\frac54\sqrt t$ as $t\ge 5$, and $\Delta^3\sqrt t>20$ as $\Delta\ge 4$. By Theorem~\ref{thm:polyorder}, \[
   |V(G)|<\Bigl(\frac54\,\Delta^3\sqrt t+1\Bigr)^2<\Bigl(\frac{13}{10}\,\Delta^3\sqrt t\Bigr)^2<2t\Delta^6<2t\,(5c^3t^2)^6=31250\,c^{18}t^{13}.
\]
 This completes the proof of Theorem~\ref{thm:main}.
\end{proof}

\section{Concluding remarks}\label{sec:remarks}
For $t\ge 5$, let $\Delta^*(t)$ denote the supremum of $\Delta(G)$ over all
connected graphs $G$ with no $K_t$ minor, minimum degree $\delta(G)\ge t-1$,
and positive Lin--Lu--Yau curvature on every edge. By
Theorem~\ref{thm:main},
\[
\Delta^*(t)<5c(t)^3t^2=O\bigl(t^5\log^{3/2}t\bigr).
\]
It would be interesting to determine whether $\Delta^*(t)=O(t)$. In particular, one may ask whether $\Delta^*(5)\le 17$, as in
Theorem~\ref{thm:LW}.
By Wagner's Theorem~\cite{Wagner}, every finite $4$-connected graph with no $K_5$ minor is planar. By Theorem~\ref{thm:LW}, to determine whether
$\Delta^*(5)\le 17$, it suffices to consider graphs of connectivity at most three.\bigskip

A graph $H$ is a \emph{topological minor} of a graph $G$ if $G$ contains a subdivision of $H$ as a subgraph. Every topological minor of $G$ is also a minor of $G$, so every graph with no $K_t$ minor has no topological $K_t$ minor.   Graphs with no topological $K_t$ minor also have only linearly many edges due to Mader~\cite{Mader68},  Bollob\'as and Thomason~\cite{BT98} and Koml\'os and Szemer\'edi~\cite{KS96}. A deep result of  Mader~\cite{Mader98} asserts that  every graph on $n\ge 5$ vertices with at least $3n-5$ edges contains a  topological  $K_5$ minor. However, the proof of Theorem~\ref{thm:main} uses contractions, most importantly in the proof of Lemma~\ref{lem:pairs} and in Claim~\ref{cl:W}. Note that $K^*_{t-2, n-t+2}$ have no topological $K_t$ minor.  The minimum degree condition in the following problem cannot be weakened to $\delta(G)\ge t-2$. It is  also worth noting that the treewidth argument used in the proof of Theorem~\ref{thm:polyorder} does not extend to forbidden topological minors. Thus, to bound $|V(G)|$ in this setting, one must instead use Lemma~\ref{lem:order}.

\begin{prob}\label{prob:top}
Let $t\ge 5$. Is there a function $g(t)$ such that every connected graph $G$ with no topological $K_t$ minor, $\delta(G)\ge t-1$, and positive Lin--Lu--Yau curvature on every edge satisfies $\Delta(G)\le g(t)$, and hence $|V(G)|\le g(t)^{2g(t)(g(t)-1)}$?  
\end{prob}

\section*{Acknowledgments}
Claude Opus 5.5 assisted the author in developing a proof of
Theorem~\ref{thm:main}, following a long  discussion in which the author shared observations and tools for addressing the problem. It also provided a survey of the topic on Ricci curvature and Lin--Lu--Yau curvature for the Introduction and assisted in creating the figure. The author subsequently developed and wrote the proofs, primarily using their own mathematical ideas and reasoning, and takes full responsibility for the accuracy and content of the manuscript.

\end{document}